\documentclass{amsart}
\usepackage{amsmath}
\usepackage{amsfonts}
\usepackage{cite}

\newtheorem{theorem}{Theorem}
\theoremstyle{plain}

\newtheorem{definition}{Definition}

\newtheorem{proposition}{Proposition}
\newtheorem{remark}{Remark}

\numberwithin{equation}{section}
\input{tcilatex}

\begin{document}
\title[A study concerning Berger type deformed Sasaki metric]{Some harmonic
problems on the tangent bundle with a Berger-type deformed Sasaki metric }
\author{Murat Altunbas}
\address[M. Altunbas]{Department of Mathematics \\
Faculty of Sciences and Arts\\
Erzincan University\\
25240 Erzurum, Turkey}
\email[M. Altunbas]{maltunbas@erzincan.edu.tr}
\author{Ramazan Simsek}
\address[M. Simsek]{Technical Sciences Vocational School,\\
Bayburt University, \\
69000, Bayburt-Turkey.}
\email[R. Simsek]{ rsimsek@bayburt.edu.tr}
\author{Aydin Gezer*}
\address[A. Gezer]{Department of Mathematics\\
Faculty of Science\\
Ataturk University\\
25240, Erzurum-Turkey}
\email[A. Gezer]{aydingzr@gmail.com}
\date{}
\thanks{*Corresponding author.}

\begin{abstract}
Let $(M_{2k},\varphi ,g)$ be an almost anti-paraHermitian manifold and $%
(TM,g_{BS})$ be its tangent bundle with a Berger type deformed Sasaki metric 
$g_{BS}$. In this paper, we deal with the harmonicity of the canonical
projection $\pi :TM\rightarrow M$ and a vector field $\xi $ which is
considered as a map $\xi :M\rightarrow TM.$


\end{abstract}

\maketitle

\section{Introduction}

Let $M$ be an $2k-$dimensional Riemannian manifold with Riemannian metric $g$%
. Throughout the paper, manifolds, tensor fields and connections are always
assumed to be differentiable of class $C^{\infty }$.

An almost paracomplex manifold is an almost product manifold $%
(M_{2k},\varphi )$, $\varphi ^{2}=id$, such that the two eigenbundles $%
T^{+}M $ and $T^{-}M$ associated to the two eigenvalues $+1$ and $-1$ of $%
\varphi $, respectively, have the same rank. The integrability of an almost
paracomplex structure is equivalent to the vanishing of the Nijenhuis tensor 
\begin{equation*}
N_{\varphi }(X,Y)=[\varphi X,\varphi Y]-\varphi \lbrack \varphi X,Y]-\varphi
\lbrack X,\varphi Y]+[X,Y].
\end{equation*}%
A paracomplex structure is an integrable almost paracomplex structure.

Let $(M_{2k},\varphi )$ be an almost paracomplex manifold. A Riemannian
metric $g$ is said to be an anti-paraHermitian metric if 
\begin{equation*}
g(\varphi X,\varphi Y)=g(X,Y)
\end{equation*}%
or equivalently 
\begin{equation*}
g(\varphi X,Y)=g(X,\varphi Y)
\end{equation*}%
for any $X,Y$ $\in \chi (M_{2k})$. If $(M_{2k},\varphi )$ is an almost
paracomplex manifold with an anti-paraHermitian metric $g$, then the triple $%
(M_{2k},\varphi ,g)$ is said to be an almost anti-paraHermitian manifold.
Moreover, $(M_{2k},\varphi ,g)$ is said to be an anti-paraK\"{a}hler if $%
\varphi $ is parallel with respect to the Levi-Civita connection $\nabla
^{g} $ of $g$. As is well known, the anti-paraK\"{a}hler condition ($\nabla
^{g}\varphi =0$) is equivalent to paraholomorphicity of the
anti-paraHermitian metric $g$, that is, ${\Phi }_{\varphi }g=0$, where ${%
\Phi }_{\varphi }$ is the Tachibana operator \cite{Salimov}.

In \cite{Altunbas}, the authors defined a new metric, which is called a
Berger type deformed Sasaki metric, on the tangent bundle over an anti-paraK%
\"{a}hler manifold. They studied the geodesics and curvature properties of
the tangent bundle with Berger type deformed Sasaki metric and gave the
conditions for some almost anti-paraHermitian structures to be anti-paraK%
\"{a}hler and quasi-anti-paraK\"{a}hler on this setting. Motivated by the
results presented in \cite{Oniciuc}, we think up the paper. Clearly, in the
present paper, we again consider the tangent bundle with the Berger type
deformed Sasaki metric. First, we deal with the natural projection $\pi
:TM\rightarrow M$ and find conditions under which $\pi $ is a totally
geodesic or a harmonic map. Second, we consider a vector field on $M$ as a
map $\xi :M\rightarrow TM$ and obtain conditions under which $\xi $ is an
isometric immersion, a totally geodesic or a harmonic map.

\section{The Berger type deformed Sasaki metric on the tangent bundle}

\noindent Let $M$ be an $n-$dimensional Riemannian manifold with a
Riemannian metric $g$ and\ $TM$ be its tangent bundle denoted by $\pi
:TM\rightarrow M$. The tangent bundle $TM$ is a $2n-$dimensional manifold. A
system of local coordinates $(U,x^{i})$ in $M$ induces on $TM$ a system of
local coordinates $\left( \pi ^{-1}\left( U\right) ,x^{i},x^{\overline{i}%
}=u^{i}\right) ,$ $\overline{i}=n+i=n+1,...,2n$, where $(u^{i})$ is the
cartesian coordinates in each tangent space $T_{P}M$ at $P\in M$ with
respect to the natural base $\left\{ {\frac{\partial }{\partial x^{i}}}%
\left\vert _{P}\right. \right\} $, $P$ being an arbitrary point in $U$ whose
coordinates are $(x^{i})$. Summation over repeated indices is always implied.

Denote by $\nabla $ the Levi-civita connection of $g.$ Local frame in
horizontal distribution $HTM$, which is determined by $\nabla $, is given by 
\begin{equation*}
E_{i}=\frac{\partial }{\partial x^{i}}-\Gamma _{i0}^{h}\frac{\partial }{%
\partial u^{h}};\text{ }i=1,...,n,
\end{equation*}%
where $\Gamma _{i0}^{h}=u^{j}\Gamma _{ij}^{h}$ are the Christoffel symbols
of $\nabla $. Local frame in vertical distribution $VTM$ which is defined by 
$\ker \pi _{\ast }$ is as 
\begin{equation*}
E_{\overline{i}}=\frac{\partial }{\partial u^{i}};\text{ }\overline{i}%
=n+1,...,2n.
\end{equation*}%
The local frame $\{E_{\beta }\}=(E_{i},E_{\overline{i}})$ in $TM$ is called
an adapted frame. With respect to the adapted frame, the horizontal lift and
vertical lift of a vector field $X=\xi ^{i}\frac{\partial }{\partial x^{i}}$ 
$\in \chi (M)$ is respectively as follows \cite{YanoIshihara:DiffGeo}: 
\begin{equation*}
^{H}X=\xi ^{i}E_{i}\in \chi (\pi ^{-1}(U))
\end{equation*}%
\begin{equation*}
^{V}X=\xi ^{i}E_{\overline{i}}\in \chi (\pi ^{-1}(U)).
\end{equation*}%
The system of local 1-forms $(dx^{i},\delta u^{i})$ in $TM$ defines the
local dual frame of the adapted frame $\{E_{\beta }\}$, where%
\begin{equation*}
\delta u^{i}=^{H}(dx^{i})=du^{i}+\Gamma _{h0}^{i}dx^{h}.
\end{equation*}%
Via natural lifts of a Riemannian metric $g$, new (pseudo) Riemannian
metrics can be induced on the tangent bundle $TM$ over a Riemannian manifold 
$(M,g)$. The well-known example of such metrics is the Sasaki metric. This
metric on the tangent bundle $TM$ was constructed by Sasaki in \cite{Sasaki}
and later studied its interesting geometric properties by some authors. Also
the various deformations of the Sasaki metric are considered and studied.
One of them is Berger type deformed Sasaki metric.

\begin{definition}
Let $(M_{2k},\varphi ,g)$ be an almost anti-paraHermitian manifold and $TM$
be its tangent bundle. The Berger type deformed Sasaki metric on $TM$ is
defined by%
\begin{eqnarray}
{}g_{BS}({}^{H}X,{}^{H}Y) &=&{}g(X,Y),  \label{1} \\
g_{BS}({}^{V}X,{}^{H}Y) &=&{}g_{BS}({}^{H}X,{}^{V}Y)=0,  \notag \\
g_{BS}({}^{V}X,{}^{V}Y) &=&{}g(X,Y)+\delta ^{2}g(X,\varphi u)g(Y,\varphi u) 
\notag
\end{eqnarray}%
for all $X,Y$ $\in \chi (M_{2k})$, where $\delta $ is some constant (if we
consider the structure $\varphi $ as almost complex, we get the metric in 
\cite{Yampolski}).
\end{definition}

The matrix of the Berger type deformed Sasaki metric with respect to the
adapted frame $\{E_{\beta }\}$ is as follows $\ $%
\begin{equation}
{}g_{BS}=\left( 
\begin{array}{cc}
g_{ij} & 0 \\ 
0 & g_{ij}+\delta ^{2}g_{m0}g_{n0}\varphi _{i}^{m}\varphi _{j}^{n}%
\end{array}%
\right)  \label{2}
\end{equation}%
and its inverse 
\begin{equation}
{}\ g_{BS}^{-1}=\left( 
\begin{array}{cc}
g^{ij} & 0 \\ 
0 & g^{ij}-\frac{\delta ^{2}}{1+\delta ^{2}g_{00}}\varphi _{0}^{i}\varphi
_{0}^{j}%
\end{array}%
\right) .  \label{3}
\end{equation}%
where $g_{m0}=g_{mk}u^{k},$ $g_{00}=g_{mk}u^{m}u^{k},$ $\varphi
_{0}^{i}=\varphi _{m}^{j}u^{m}$. For the Levi-Civita connection $^{BS}\nabla 
$ of the Berger type deformed Sasaki metric $g_{BS}$, we give the following
proposition.

\begin{proposition}
\cite{Altunbas} Let $(M_{2k},\varphi ,g)$ be an anti-paraK\"{a}hler manifold
and $TM$ its tangent bundle. The Levi-Civita connection $^{BS}\nabla $ of
the Berger type deformed Sasaki metric $g_{BS}$ on $TM$ is given locally by%
\begin{equation}
\left\{ 
\begin{array}{l}
^{BS}\nabla _{{}E_{i}}{}E_{j}=\Gamma _{ij}^{h}E_{h}-\frac{1}{2}R_{ij0}^{%
\text{ \ \ }h}E_{\bar{h}}, \\ 
^{BS}\nabla _{{}E_{\overline{i}}}{}E_{j}=-\frac{1}{2}R_{.j0i}^{h\text{ \ }%
.}E_{h}, \\ 
^{BS}\nabla _{{}E_{i}}{}E_{\bar{j}}=-\frac{1}{2}R_{.i0j}^{h\text{ \ }%
.}E_{h}+\Gamma _{ij}^{h}E_{\bar{h}}, \\ 
^{BS}\nabla _{{}E_{\overline{i}}}{}E_{\bar{j}}=\frac{\delta ^{2}}{1+\delta
^{2}}\varphi _{j}^{k}\varphi _{0}^{h}g_{ik}E_{\bar{h}},%
\end{array}%
\right.   \label{4}
\end{equation}%
where $R_{ijk}^{\text{ \ \ }h}$ are the local components of the Riemannian
curvature tensor field $R$ of $\nabla $ and $R_{.j0i}^{h\text{ \ }%
.}=R_{lj0}^{\text{ \ \ }s}g^{lh}g_{si}$.
\end{proposition}

\section{Main Results}

Let $M$ and $N$ be two Riemannian manifolds, $U\subset M$ be a domain with
coordinates $(x^{1},...,x^{m})$ and $V\subset N$ be a domain with
coordinates $(u^{1},...,u^{n})$. Also, let $f$ $:U\rightarrow V$ be a map in
which $u^{\alpha }=f^{\alpha }(x^{1},...,x^{m})$. The second fundamental
form of $f$ denoted by $\beta (f)$ is%
\begin{equation}
\beta (f)(\frac{\partial }{\partial x^{i}},\frac{\partial }{\partial x^{j}}%
)^{\gamma }=\{\frac{\partial ^{2}f^{\gamma }}{\partial x^{i}\partial x^{j}}%
-^{M}\Gamma _{ij}^{k}\frac{\partial f^{\gamma }}{\partial x^{k}}+^{N}\Gamma
_{\alpha \beta }^{\gamma }\frac{\partial f^{\alpha }}{\partial x^{i}}\frac{%
\partial f^{\beta }}{\partial x^{j}}\}\frac{\partial }{\partial u^{\gamma }}
\label{5}
\end{equation}%
and that of the tension field $\tau (f)$ of $f$ is%
\begin{equation}
\tau (f)=g^{ij}\{\frac{\partial ^{2}f^{\gamma }}{\partial x^{i}\partial x^{j}%
}-^{M}\Gamma _{ij}^{k}\frac{\partial f^{\gamma }}{\partial x^{k}}+^{N}\Gamma
_{\alpha \beta }^{\gamma }\frac{\partial f^{\alpha }}{\partial x^{i}}\frac{%
\partial f^{\beta }}{\partial x^{j}}\}\frac{\partial }{\partial u^{\gamma }}%
\text{ \ \ \ (see \cite{Eells}).}  \label{6}
\end{equation}

Remark that $\beta (f)$ and $\tau (f)$ may also be defined when $N$ is
endowed with a torsion free linear connection or $M$ is a pseudo-Riemannian
manifold.

We note that the canonical projection $\pi :(TM,{}g_{BS})\rightarrow
(M_{2k},\varphi ,g)$ is a Riemannian submersion. Direct computations give%
\begin{eqnarray}
\beta (\pi )(E_{i},E_{j}) &=&\beta (\pi )(E_{\overline{i}},E_{\bar{j}})=0,
\label{7} \\
\beta (\pi )_{v}(E_{\overline{i}},E_{j}) &=&\frac{1}{2}R_{.j0i}^{h\text{ \ }%
.}(\pi (v))\frac{\partial }{\partial x^{h}}.  \notag
\end{eqnarray}%
From (\ref{7}), we obtain the following theorem.

\begin{theorem}
Let $(M_{2k},\varphi ,g)$ be an anti-paraK\"{a}hler manifold and $%
(TM,g_{BS}) $ be its tangent bundle with the Berger type deformed Sasaki
metric. The Riemannian submersion $\pi :(TM,{}g_{BS})\rightarrow
(M_{2k},\varphi ,g)$ is totally geodesic if and only if $\nabla $ is locally
flat. Moreover $\pi $ is a harmonic map.
\end{theorem}

Let $h$ be another anti-paraHermitian metric on $M$ with respect to an
almost paracomplex structure $\varphi _{1}$. We take in consideration the
projection $\pi :(TM,{}g_{BS})\rightarrow (M_{2k},\varphi _{1},h)$. Then we
have 
\begin{eqnarray}
\beta (\pi )(E_{\overline{i}},E_{\bar{j}}) &=&0,  \label{8} \\
\beta (\pi )_{v}(E_{\overline{i}},E_{j}) &=&\frac{1}{2}R_{.j0i}^{h\text{ \ }%
.}(\pi (v))\frac{\partial }{\partial x^{h}},  \notag \\
\beta (\pi )(E_{i},E_{j}) &=&\{^{h}\Gamma _{ij}^{h}-\Gamma _{ij}^{h}\}\frac{%
\partial }{\partial x^{h}},  \notag
\end{eqnarray}%
where $^{h}\Gamma _{ij}^{h}$ are the Christoffel symbols of the metric $h.$
Hence we get the proposition below.

\begin{proposition}
Let $(M_{2k},\varphi ,g)$ be an anti-paraK\"{a}hler manifold and $%
(TM,g_{BS}) $ be its tangent bundle with the Berger type deformed Sasaki
metric. $\pi :(TM,{}g_{BS})\rightarrow (M_{2k},\varphi _{1},h)$ is totally
geodesic if and only if $(M_{2k},\varphi ,g)$ is locally flat and $%
I:(M_{2k},\varphi ,g)\rightarrow (M_{2k},\varphi _{1},h)$ is totally
geodesic.
\end{proposition}

Note that if $\pi :(TM,{}g_{BS})\rightarrow (M_{2k},\varphi _{1},h)$ is
totally geodesic, then $(M_{2k},\varphi ,g)$ and $(M_{2k},\varphi _{1},h)$
are locally flat.

Let $g$ and $h$ be two anti-paraHermitian metrics on $M$. It is said that $h$
is harmonic with respect to $g$ if 
\begin{equation}
g^{ij}(^{h}\Gamma _{ij}^{h}-\Gamma _{ij}^{h})=0\text{ (see \cite{Chen}).}
\label{9}
\end{equation}%
The equations in (\ref{8}) give the following proposition.

\begin{proposition}
Let $(M_{2k},\varphi ,g)$ be an anti-paraK\"{a}hler manifold and $%
(TM,g_{BS}) $ be its tangent bundle with the Berger type deformed Sasaki
metric. $\pi :(TM,{}g_{BS})\rightarrow (M_{2k},\varphi _{1},h)$ is harmonic
if and only if $h$ is harmonic with respect to $g.$
\end{proposition}

Now we consider a vector field $\xi $ as $\xi :(M_{2k},\varphi
,g)\rightarrow (TM,{}g_{BS})$. First, we should prove the proposition below.

\begin{proposition}
\label{PropositionA}Let $(M_{2k},\varphi ,g)$ be an anti-paraK\"{a}hler
manifold and $(TM,g_{BS})$ be its tangent bundle with the Berger type
deformed Sasaki metric. $\xi :(M_{2k},\varphi ,g)\rightarrow (TM,{}g_{BS})$
is an isometric immersion if and only if $\nabla \xi =0.$
\end{proposition}

\begin{proof}
The following equation holds%
\begin{equation}
\xi _{\ast ,p}X=(^{H}X+^{V}(\nabla _{X}\xi ))_{\xi (p)},\text{ \ }\forall
p\in M.  \label{10}
\end{equation}%
If we suppose 
\begin{eqnarray*}
\overset{1}{g}_{p}(X,Y) &=&g_{BS_{\xi (p)}}(\xi _{\ast ,p}X,\xi _{\ast
,p}Y)=g_{p}(X,Y)+g_{p}(\nabla _{X}\xi ,\nabla _{Y}\xi ) \\
&&+\delta ^{2}g_{p}(\nabla _{X}\xi ,\varphi \xi )g_{p}(\nabla _{Y}\xi
,\varphi \xi ),
\end{eqnarray*}%
then $\xi $ is an isometric immersion if and only if $\overset{1}{g}=g.$ It
is obvious that $\overset{1}{g}=g$ if and only if $\nabla _{X}\xi =0$ for
all $X\in \chi (M)$, i.e., $\nabla \xi =0.$
\end{proof}

By using the equations (\ref{4}) and (\ref{10}) we get the proposition below.

\begin{proposition}
Let $(M_{2k},\varphi ,g)$ be an anti-paraK\"{a}hler manifold and $%
(TM,g_{BS}) $ be its tangent bundle with the Berger type deformed Sasaki
metric. The second fundamental form $\beta (\xi )$ and the tension field $%
\tau (\xi )$ of the map $\xi :(M_{2k},\varphi ,g)\rightarrow (TM,{}g_{BS})$
are given by%
\begin{eqnarray}
\beta (\xi )(\frac{\partial }{\partial x^{i}},\frac{\partial }{\partial x^{j}%
}) &=&-\frac{1}{2}\{(\nabla _{j}\xi ^{k})R_{.ikm}^{h\text{ \ }.}\xi
^{m}+(\nabla _{i}\xi ^{l})R_{.jlm}^{h\text{ \ }.}\xi ^{m}\}E_{h}  \label{11}
\\
&&+\{-\frac{1}{2}R_{ijm}^{h}\xi ^{m}+\nabla _{i}\nabla _{j}\xi ^{h}+(\nabla
_{i}\xi ^{m})(\nabla _{j}\xi ^{n})A_{mn}^{h}\}E_{\overline{h}},  \notag
\end{eqnarray}%
\begin{equation}
\tau (\xi )=\{g^{ij}(\nabla _{j}\xi ^{k})R_{.ikm}^{h\text{ \ }.}\xi
^{m}\}E_{h}+\{g^{ij}(\nabla _{i}\nabla _{j}\xi ^{h})+g^{ij}(\nabla _{i}\xi
^{m})(\nabla _{j}\xi ^{n})A_{mn}^{h}\}E_{\overline{h}},  \label{12}
\end{equation}%
where $A_{mn}^{h}=\frac{\delta ^{2}}{1+\delta ^{2}}\varphi _{n}^{k}\varphi
_{0}^{h}g_{mk}$.
\end{proposition}

The equation (\ref{12}) gives the following theorem.

\begin{theorem}
Let $(M_{2k},\varphi ,g)$ be an anti-paraK\"{a}hler manifold and $%
(TM,g_{BS}) $ be its tangent bundle with the Berger type deformed Sasaki
metric. The map $\xi :(M_{2k},\varphi ,g)\rightarrow (TM,{}g_{BS})$ is
harmonic if and only if 
\begin{equation}
g^{ij}(\nabla _{j}\xi ^{k})R_{.ikm}^{h\text{ \ }.}\xi ^{m}=0,\text{ \ \ \ }%
g^{ij}(\nabla _{i}\nabla _{j}\xi ^{h})+g^{ij}(\nabla _{i}\xi ^{m})(\nabla
_{j}\xi ^{n})A_{mn}^{h}=0.  \label{13}
\end{equation}
\end{theorem}

As a direct consequence of Proposition \ref{PropositionA} and (\ref{11}), we
obtain the theorem below.

\begin{theorem}
Let $(M_{2k},\varphi ,g)$ be an anti-paraK\"{a}hler manifold and $%
(TM,g_{BS}) $ be its tangent bundle with the Berger type deformed Sasaki
metric. If $\xi :(M_{2k},\varphi ,g)\rightarrow (TM,{}g_{BS})$ is isometric
immersion, then $\xi $ is totally geodesic. Furthermore, $\xi $ is harmonic.
\end{theorem}

\begin{remark}
$\xi $ is parallel if and only if $\xi $ is a minimal isometric immersion.
\end{remark}

Now we investigate the harmonicity of the Berger type deformed Sasaki metric 
$g_{BS}$ and the Sasaki metric $g_{S}$ with respect to each other. By using
the Christoffel symbols of these metrics we state the following two
propositions (for the Christoffel symbols of the Sasaki metric see \cite%
{YanoIshihara:DiffGeo}).

\begin{proposition}
Let $(M_{2k},\varphi ,g)$ be an anti-paraK\"{a}hler manifold and $%
(TM,g_{BS}) $ be its tangent bundle with the Berger type deformed Sasaki
metric. Suppose that $I:(TM,g_{BS})\rightarrow (TM,g_{S})$ is the identity
map. Then the followings hold:

1) $I$ cannot be totally geodesic;

2) The tension field $\tau _{g_{BS}}(I)$ of $I$ is given by%
\begin{equation}
\text{\ }\tau _{g_{BS}}(I_{TM})=(-\frac{\delta ^{2}n}{1+\delta ^{2}}\varphi
_{0}^{h}+\frac{\delta ^{4}g_{00}u^{h}}{(1+\delta ^{2})(1+\delta ^{2}g_{00})}%
)E_{\bar{h}}.  \label{14}
\end{equation}%
So, $g_{S}$ cannot be harmonic with respect to $g_{BS}$.
\end{proposition}

\begin{proposition}
Let $(M_{2k},\varphi ,g)$ be an anti-paraK\"{a}hler manifold and $%
(TM,g_{BS}) $ be its tangent bundle with the Berger type deformed Sasaki
metric. Suppose that $I:(TM,g_{S})\rightarrow (TM,g_{BS})$ is the identity
map. Then the tension field $\tau _{g_{S}}(I)$ of $I$ is given by 
\begin{equation}
\text{\ }\tau _{g_{BS}}(I_{TM})=(\frac{\delta ^{2}n}{1+\delta ^{2}}\varphi
_{0}^{h})E_{\bar{h}}.  \label{15}
\end{equation}
\end{proposition}

Finally we study the conditions for the vanishing of the second fundamental
form of the identity map between $(TM,g_{BS})$ and $(TM,^{BS}\overset{m}{%
\tilde{\nabla}})$, where $^{BS}\overset{m}{\tilde{\nabla}}$ is the mean
connection of the Schouten-Van Kampen connection associated with the
Levi-Civita connection of the Berger type deformed Sasaki metric$.$ The
Schouten-Van Kampen connection $^{BS}\tilde{\nabla}$ associated with $%
^{BS}\nabla $ is defined by 
\begin{equation}
^{BS}\tilde{\nabla}_{\tilde{X}}\tilde{Y}=V^{BS}\nabla _{\tilde{X}}V\tilde{Y}%
+H^{BS}\nabla _{\tilde{X}}H\tilde{Y},\text{ \ }  \label{16}
\end{equation}%
where $\tilde{X}$ and $\tilde{Y}$ $\in \chi (TM),$ and $V$ and $H$ are
vertical and horizontal projections. The local components of $^{BS}\tilde{%
\nabla}_{\tilde{X}}\tilde{Y}$ are as follows:

\begin{equation}
\left\{ 
\begin{array}{l}
^{BS}\tilde{\nabla}_{{}E_{i}}{}E_{j}=\Gamma _{ij}^{h}E_{h}, \\ 
^{BS}\tilde{\nabla}_{{}E_{\overline{i}}}{}E_{j}=^{BS}\nabla _{{}E_{\bar{%
\imath}}}{}E_{j}, \\ 
^{BS}\widetilde{\nabla }_{{}E_{i}}{}E_{\bar{j}}=\Gamma _{ij}^{h}E_{\bar{h}},
\\ 
^{BS}\widetilde{\nabla }_{{}E_{\overline{i}}}{}E_{\bar{j}}=^{BS}\nabla
_{{}E_{\bar{\imath}}}{}E_{\bar{j}}.%
\end{array}%
\right.   \label{17}
\end{equation}%
The mean connection of $^{BS}\tilde{\nabla}$ is $^{BS}\overset{m}{\tilde{%
\nabla}}=$ $^{BS}\tilde{\nabla}-\frac{1}{2}T$, where $T$ is the torsion
tensor field of $^{BS}\tilde{\nabla}$. For the local components of $^{BS}%
\overset{m}{\tilde{\nabla}}$, we find%
\begin{equation}
\left\{ 
\begin{array}{l}
^{BS}\overset{m}{\tilde{\nabla}}_{{}E_{i}}{}E_{j}=^{BS}\tilde{\nabla}%
_{{}E_{i}}{}E_{j}, \\ 
^{BS}\overset{m}{\tilde{\nabla}}_{{}E_{\overline{i}}}{}E_{j}=\frac{1}{2}%
^{BS}\nabla _{{}E_{\bar{\imath}}}{}E_{j}, \\ 
^{BS}\overset{m}{\tilde{\nabla}}_{{}E_{i}}{}E_{\bar{j}}=^{BS}\nabla
_{{}E_{i}}{}E_{\bar{j}}-\frac{1}{4}R_{.i0j}^{h\text{ \ }.}E_{h}, \\ 
^{BS}\overset{m}{\tilde{\nabla}}_{{}E_{\overline{i}}}{}E_{\bar{j}%
}=^{BS}\nabla _{{}E_{\bar{\imath}}}{}E_{\bar{j}}.%
\end{array}%
\right.   \label{18}
\end{equation}%
It is clear that $^{BS}\overset{m}{\tilde{\nabla}}$ is a torsion-free
connection. Moreover, note that $^{BS}\overset{m}{\tilde{\nabla}}=^{BS}%
\tilde{\nabla}$ if and only if $R=0$. The following proposition ends the
paper.

\begin{proposition}
Let $(M_{2k},\varphi ,g)$ be an anti-paraK\"{a}hler manifold and $(TM,g_{BS})
$ be its tangent bundle with the Berger type deformed Sasaki metric. The
second fundamental form of the map $I:(TM,g_{BS})\rightarrow (TM,^{BS}%
\overset{m}{\tilde{\nabla}})$ satisfies%
\begin{equation}
\beta (I)(E_{i},E_{j})=\beta (I)(E_{\overline{i}},E_{\bar{j}})=0,\backslash
,\backslash ,\beta (I)(E_{\overline{i}},E_{j})=\frac{1}{4}R_{.j0i}^{h\text{
\ }.}E_{h}.  \label{19}
\end{equation}%
Hence $\tau (I)=\mathrm{trace}\beta (I)=0.$ Moreover if $R\neq 0$ then $%
\beta (I)\neq 0.$
\end{proposition}

\end{document}